\documentclass[12pt,leqno]{article}

\usepackage[lite]{amsrefs}
\usepackage{latexsym,enumerate}

\usepackage{amssymb, amsfonts, eucal,amsmath,amsthm}
\newtheorem{theorem}{Theorem}[section]

\newtheorem{sametheorem}{Theorem}

\newtheorem{lemma}[theorem]{Lemma}
\newtheorem{proposition}[theorem]{Proposition}
\newtheorem{remark}[theorem]{Remark}
\newtheorem{definition}[theorem]{Definition}

\newtheorem{corollary}[theorem]{Corollary}

\newcommand{\sect}[1]{\section{#1} \setcounter{equation}{0} }

\newcounter{ca}
\newcommand{\case}[1]{\vspace{0.5em} \noindent {\scshape Case \theca:} \addtocounter{ca}{1}  #1  \vspace{0.5em}}

\newcommand{\norm}[2]{\left\|#1\right\|_{#2}}

\newcommand{\ds}{\displaystyle}

\newcommand{\Pn}{\mathbb P_n}

 \newcommand{\ec}{\end{comment}}
\newcommand{\bc}{ \begin{comment}
 }

\newcommand{\F}{{\mathcal F}}
\newcommand{\E}{{\mathcal E}}

\newcommand{\A}{{\mathcal A}}

\newcommand{\andd}{\quad\mbox{\rm and}\quad}

\newcommand\e{{\varepsilon}}

\newcommand\w{{\omega}}
\def\be  {\begin{equation}}
\def\ee  {\end{equation}}
\def\ba  {\begin{eqnarray}}
\def\ea  {\end{eqnarray}}
\def\baa {\begin{eqnarray*}}
\def\eaa {\end{eqnarray*}}
\newenvironment{comment}[2]
{\bgroup\vspace{7pt}
\begin{tabular}{|p{5in}|}
\hline \qquad \bf \footnotesize Comment -- to be deleted in the final version \\
\hline
\quad\sl\footnotesize #1#2} {\\ \hline \end{tabular}
\vspace{7pt}\indent\egroup}

\def\updots{\mathinner{\mkern
1mu\raise 1pt \hbox{.}\mkern 2mu \mkern 2mu \raise
4pt\hbox{.}\mkern 1mu \raise 7pt\vbox {\kern 7 pt\hbox{.}}} }

\def \meas{\mathop{\rm meas}\nolimits}

\def \const{\mathop{\rm const}\nolimits}

\newcommand{\tilt}{\tilde t}

\newcommand{\B}{\mathbb B}
\newcommand{\C}{C}

\newcommand{\R}{\mathbb R}

\newcommand{\N}{\mathbb N}

 \renewcommand{\L}{\mathbb L}

\newcommand{\ineq}[1]{{\rm(\ref{#1})}}

\newcommand{\ie}{{\em i.e., }}
\newcommand{\eg}{{\em e.g., }}

\newcommand{\bpic}{
\begin{center}
}

\newcommand{\epic}{
\endpspicture
\end{center}
}

\newcommand{\st}{\;\; \big| \;\;}

\renewcommand{\L}{\mathbb L}
\newcommand{\Lp}{\L_p}
\newcommand{\Lq}{\L_q}

 \newcommand{\AC}{\mathrm{AC}}
  \newcommand{\loc}{\mathrm{loc}}
\newcommand{\wkr}{\w_{k,r}^\varphi}
\newcommand{\wkrav}{\w_{k,r}^{*\varphi}}
 \newcommand{\Dom}{{\mathfrak{D}}}
  \newcommand{\Ddom}{\mathrm{Dom}}

\newcommand{\ddelta}{\mu(\delta)}

 \newcommand{\wt}{{\mathcal{W}}}

\newcommand{\thm}[1]{Theorem~\ref{#1}}
\newcommand{\lem}[1]{Lemma~\ref{#1}}
\newcommand{\cor}[1]{Corollary~\ref{#1}}
\newcommand{\prop}[2]{Proposition~\ref{#1}(\ref{#2})}

\title{{\sc New moduli of smoothness: weighted DT moduli revisited and applied}
\thanks{{\it AMS classification:} 41A10, 41A17, 41A25. {\it Keywords
and phrases:} Approximation by polynomials in the $L_p$-norm, Degree
of approximation, Jackson-type estimates, moduli of smoothness. }
\thanks{Part of this work was done while the first two authors were
at the Centre de Recerca Matem\`atica, Barcelona.} }

\date{April 21, 2014}

\author{
K. A.  Kopotun\thanks{Department of Mathematics, University of
Manitoba, Winnipeg, Manitoba, R3T 2N2, Canada ({\tt
kopotunk@cc.umanitoba.ca}). Supported by NSERC of Canada.} , D.
Leviatan \thanks{Raymond and Beverly Sackler School of Mathematical
Sciences, Tel Aviv University, Tel Aviv 69978, Israel ({\tt
leviatan@post.tau.ac.il}).} \ and I. A. Shevchuk
\thanks{Faculty of Mechanics and Mathematics, National Taras
Shevchenko University of Kyiv, 01033 Kyiv, Ukraine ({\tt
shevchuk@univ.kiev.ua}).} }

\begin{document}

\maketitle

\abstract{We introduce new moduli of smoothness for functions
$f\in L_p[-1,1]\cap C^{r-1}(-1,1)$, $1\le p\le\infty$,
$r\ge1$, that have an $(r-1)$st locally absolutely continuous derivative in $(-1,1)$, and such that
$\varphi^rf^{(r)}$ is in $L_p[-1,1]$, where
$\varphi(x)=(1-x^2)^{1/2}$. These moduli are equivalent to certain
weighted DT moduli, but our definition is more transparent and
simpler. In addition, instead of applying these weighted moduli to
weighted approximation, which was the purpose of the original DT
moduli, we apply these moduli to obtain Jackson-type estimates on
the approximation of functions in $L_p[-1,1]$ (no weight), by means
of algebraic polynomials.
Moreover, we also prove matching inverse theorems thus obtaining
constructive characterization of various smoothness classes of functions via the degree of their approximation by algebraic polynomials.}

\sect{Motivation}

The purpose of this section is to provide some motivation to the introduction of the new moduli of smoothness that we discuss in this paper.

We start with a  simple example. Suppose that  $A_p^\alpha$ is the space of all functions in $\Lp[-1,1]$, $1\leq p \leq \infty$, such that their rate of approximation by algebraic polynomials of degree $<n$ in the $\Lp$-norm is $O(n^{-\alpha})$. How can we characterize this approximation space? The answer is very well known by now. There are several approaches but the ones that became most popular in recent decades involve moduli of smoothness of Ivanov
$\tau_k(f, \delta(t,\cdot))_{q,p}$ (introduced in 1980-1981) and Ditzian-Totik $\omega_k^\varphi(f,t)_p$ (introduced around 1984). The Ditzian-Totik (DT) modulus is defined in \ineq{wkrdefinition} by letting  $r=0$ (see also Remark~\ref{remarkdt}), and the Ivanov modulus (see \cite[Section 16]{d20}, for example) is given by
\[
\tau_k(f, \delta(t,\cdot))_{q,p} = \left\|   \omega_k(f, \cdot, \delta(t,\cdot))_q \right\|_p \, ,
\]
where
\[
\omega_k(f, x, \delta(t,x))_q^q = \frac{1}{2\delta(t,x)} \int_{-\delta(t,x)}^{\delta(t,x)}
\left| \Delta_{\nu}^{k}(f, x) \right|^q\, d\nu \, .
\]
It turns out (see \eg \cite{d20}) that
$\w_k^\varphi(f,t)_p \sim \tau_k(f, \Delta (t,\cdot))_{p,p}$  with $\Delta (t, \cdot) := t\varphi(\cdot)  +t^2$, but, according to \cite[p. 142]{d20}, ``The [Ivanov] moduli ... are a somewhat more cumbersome method to describe smoothness than ... [DT moduli], and their computation is more difficult.''

It follows from \cite[Theorems 7.2.1 and 7.2.4]{dt} that, for $0<\alpha<k$,
\be \label{eqdt}
f\in A_p^\alpha \quad \iff \quad \w_k^\varphi(f,t)_p = O(t^\alpha) , \; t>0 .
\ee

A natural question now is what can be said about smoothness of the derivatives of functions from $A_p^\alpha$. Surely, if $\alpha$ is large enough, then functions from
$A_p^\alpha$ have to be differentiable (or rather, almost everywhere, they coincide with functions which are differentiable). Note that \ineq{eqdt} does not explicitly describe the behavior of these derivatives (but see Remark~\ref{auxrem} below).
While it is true that $\w_k^\varphi(f,t)_p \leq c t^r \w_{k-r}^\varphi(f^{(r)},t)_p$, it is NOT true that, for appropriate $\alpha$, $f\in A_p^\alpha$ only if $\w_{k-r}^\varphi(f^{(r)},t)_p= O(t^{\alpha-r})$.
One needs to replace $\w_{k-r}^\varphi(f^{(r)},t)_p$ with an appropriated weighted modulus (as we show in Section~\ref{sec8}). This is very different from the trigonometric case where the classical moduli of smoothness are used in analogous results on characterization of (trigonometric) approximation spaces.

%
%

\begin{remark} \label{auxrem}
To be more precise, we mention that \ineq{eqdt} implicitly  describes   the behavior of the derivatives of $f$ since it follows from \cite[Theorem 6.2.2 and Corollary 6.3.2]{dt} (\cite[Corollary 6.3.2]{dt} has  a couple of misprints, but this can be easily rectified) that, for $1\leq r<\alpha$,
\begin{eqnarray*}
\w_k^\varphi(f,t)_p = O(t^\alpha) &\iff& \omega_{\varphi}^{k-r}(f^{(r)},t)_{\varphi^r,p} = O(t^{\alpha-r}) \\
& \iff& \Omega^{k-r}_\varphi(f^{(r)},t)_{\varphi^r,p} = O(t^{\alpha-r}) .
\end{eqnarray*}
However,  if $t^\alpha$ is replaced by a more complicated function and, correspondingly, $A_p^\alpha$ is replaced by the space of functions $A_p(\phi)$ whose rate of approximation is $O(\phi(1/n))$, then \cite[Theorem 6.2.2 and Corollary 6.3.2]{dt} may no longer provide any useful information, and it becomes much harder to get an explicit description of the behavior of the derivatives of functions from $A_p(\phi)$.
For example, if $\phi(t) := t/(\ln(t/2))^2$, then
\[
f\in A_p(\phi) \quad \iff \quad \w_2^\varphi(f,t)_p = O(\phi(t))  ,
\]
and \cite[Theorem 6.3.1 (a)]{dt} implies that $f'$ is locally absolutely continuous and $\Omega_\varphi(f',t)_{\varphi,p} = O(-1/\ln(t/2))$. However,
\cite[Theorem 6.2.2]{dt} does not give any information about
$\omega_{\varphi} (f',t)_{\varphi,p}$ other than that it is bounded below by $\Omega_\varphi(f',t)_{\varphi,p}$. Hence, there is a need for an inverse theorem for an algebraic approximation explicitly involving derivatives of functions as in the classical trigonometric case (see \cite[Theorem 6.1.3]{ti-book}, for example). We prove such a theorem  in Section~\ref{sec8} (see Theorem~{\rm\ref{Theorem 3.199}}). In particular, it implies that
\[
f\in A_p(\phi) \quad \Longrightarrow \quad \omega_{1,1}^\varphi(f',t)_p = O(-1/\ln(t/2))  .
\]
\end{remark}

In summary, there is a need for a new measure of smoothness and/or new results that would help resolve the above mentioned problems. The purpose of this paper is to introduce new moduli serving this purpose.
These moduli are equivalent to certain weighted DT moduli, but, to rephrase \cite[p. 142]{d20}, ``these weighted DT moduli are a somewhat more cumbersome method to describe smoothness than our moduli, and their computation is more difficult'' (see Section~\ref{sec5} for the exact definition of these weighted DT moduli).

\sect{Introduction and definitions}
As alluded to above, we are interested in the constructive characterization of the functions in $L_p[-1,1]$, $1\le p<\infty$ and $C[-1,1]$ when $p=\infty$, with given degree of
approximation by algebraic polynomials, which is analogous to the characterization of periodic functions in $L_p[-\pi,\pi]$, respectively, $C[-\pi,\pi]$, with given degree of
approximation by trigonometric polynomials. Our characterization yields information on the smoothness of the derivatives of the approximated functions, and is described in
Sections \ref{sec5} and \ref{sec8} by means of direct and inverse theorems relating certain weighted DT moduli of smoothness of a function $f\in L_p[-1,1]$, respectively,
$f\in C[-1,1]$, to its degrees of best unweighted approximation in the space.

The first sections are devoted to introducing the above mentioned DT moduli of smoothness in a new, equivalent form, which is more transparent and simpler. We
prove the equivalence via $K$-functionals. For $p=\infty$, these moduli of smoothness were introduced by the third author \cite{sh-book}
(see also \cite{sh}), and certain direct and inverse theorems proved, however, no relations to weighted DT moduli were discussed.

In the sequel we will have constants $c$ that may depend only on some of the parameters involved ($p$, $k$, $r$), but are independent of the function and of $t$ or $n$, as the case may be.
The constants $c$ may be different even if they appear in the same line.

Let $\|\cdot\|_p:=\|\cdot\|_{\Lp[-1,1]}$, $ 1\le p \leq \infty$, and
$\varphi(x):=\sqrt{1-x^2}$.

For $k\in\N_0$, $h\geq 0$, an interval $J$ and $f: J \mapsto\R$, let
\[\Delta_h^k(f,x, J):=\left\{
\begin{array}{ll} \ds
\sum_{i=0}^k  {\binom ki}
(-1)^{k-i} f(x+(i-k/2)h),&\mbox{\rm if }\, x\pm kh/2  \in J \,,\\
0,&\mbox{\rm otherwise},
\end{array}\right.\]
be the $k$th symmetric difference, and let $\Delta_h^k(f,x) := \Delta_h^k(f,x,[-1,1])$.
\begin{definition} Let $1\le p \leq \infty$ and $r\in\N_0$. Then
for $r\ge1$, let
$$
\B_p^r:=\{\ f\ |\ f^{(r-1)}\in
AC_{loc}(-1,1)\quad\text{and}\quad\|f^{(r)}\varphi^r\|_p<+\infty\},
$$
and set $\B_p^0:=\L_p[-1,1]$.
\end{definition}
(Recall that $AC_{loc}(-1,1)$ denotes the set of  functions which are locally absolutely continuous in $(-1,1)$.)

\begin{definition}
For $f\in \B^r_p$, define
\be \label{wkrdefinition}
\wkr(f^{(r)},t)_p := \sup_{0 \leq h \leq t} \norm{\wt^r_{kh}(\cdot)
\Delta_{h\varphi(\cdot)}^k (f^{(r
)},\cdot)}{p} ,
\ee
where
\[
\wt_\delta(x):= \bigl((1-x-\delta\varphi(x)/2)
(1+x-\delta\varphi(x)/2)\bigr)^{1/2}.
\]
\end{definition}

For $\delta >0$, denote
\begin{align*}
\Dom_\delta:=&\left\{x\st1-\delta\varphi(x)/2\geq|x|
\right\}\setminus\{\pm1\}\\=&\left\{x\st|x|\leq
\frac{4-\delta^2}{4+\delta^2}\right\}=[-1+\ddelta,1-\ddelta],
\end{align*}
where
\[
\ddelta:=2\delta^2/(4+\delta^2).
\]
Observe that $\Dom_\delta=\emptyset$ if $\delta>2$, and note that
${\Delta}_{h\varphi(x)}^k(f,x)$ is defined to be identically 0 if
$x\not\in\Dom_{kh}$ and that $ \wt_\delta$ is well defined on
$\Dom_\delta$ (in fact, if $\delta \leq 2$, then $\Ddom(\wt_\delta) = \Dom_\delta \cup \{\pm 1\}$).

Hence,
\[
\wkr(f^{(r)},t)_p=\sup_{0< h \leq t} \norm{ \wt^r_{kh}(\cdot)
\Delta_{h\varphi(\cdot)}^k (f^{(r)},\cdot)}{\Lp(\Dom_{kh})}
\]
and
\be\label{larget}
\wkr(f^{(r)},t)_p = \wkr(f^{(r)},2/k)_p, \quad \mbox{\rm for } t\geq 2/k .
\ee

\begin{remark} \label{remarkdt}
 When $r=0$, $\omega^\varphi_{k,0}(f,t)_p$ reduces to
the well known $k$th DT modulus of smoothness $\omega^\varphi_{k}(f,t)_p$ $($see, \eg
\cite{dt}$)$.
\end{remark}

\begin{remark} When $p=\infty$, $\omega^\varphi_{k,r}(f,t)_\infty$ reduces to
the   modulus of smoothness introduced by the third author $($see, \eg
\cites{sh-book,sh}$)$.
\end{remark}

Our moduli of smoothness are certain type of weighted DT moduli
(see Section \ref{sec3} for details). However, we give a more
transparent and simpler definition of the moduli, which, in
particular, makes their monotonicity in $t$, self-evident. Moreover,
we are not interested in weighted approximation, rather we are
interested in applying these moduli to estimates on the non-weighted
approximation of $f\in \B^r_p$ (see Section \ref{sec5} for
details).

If $1\leq p <\infty$, our moduli are equivalent to the following averaged moduli of smoothness.

\begin{definition} \label{avemodulus}
Let   $k\in\N$, $r\in\N_0$  and
$f\in\B^r_p$, $1\le p<\infty$. Then, the averaged modulus of smoothness is defined as
\[
\omega_{k,r}^{*\varphi}(f^{(r)},t)_p
:=\left(\frac1t\int_0^t\int_{\Dom_{k\tau}}
|\wt^r_{k\tau}(x)\Delta^k_{\tau\varphi(x)}(f^{(r)},x)|^p\,dx\,d\tau
\right)^{1/p}.
\]
\end{definition}
 For convenience, for $p=\infty$, we also define
 \[
\omega_{k,r}^{*\varphi}(f^{(r)},t)_\infty := \wkr (f^{(r)},t)_\infty .
\]
While the modulus $\wkr (f^{(r)},t)_p$ is obviously a non-decreasing function of $t$, the averaged modulus $\wkrav (f^{(r)},t)_p$ does not have to be non-decreasing. At the same time, it immediately follows from Definition~\ref{avemodulus} that
\be \label{ineqavermon}
\wkrav (f^{(r)},t_1)_p \leq \left( t_2/t_1\right)^{1/p} \wkrav (f^{(r)},t_2)_p , \quad \mbox{\rm for } 0<t_1\leq t_2 .
\ee

It turns out that the above defined moduli are equivalent to the following
$K$-functional.
\begin{definition}[$K$-functional]\label{k-func} For $k\in\N$, $r\in\N_0$, $1\leq p\leq \infty$ and $f\in\B_p^r$, denote
$$
K^\varphi_{k,r}(f^{(r)},t^k)_p:=\inf_{g\in\B_p^{k+r}}(\|(f^{(r)}-g^{(r)})\varphi^r\|_p
+t^k\|g^{(k+r)}\varphi^{k+r}\|_p).
$$
\end{definition}

The following result is valid.
\begin{theorem}\label{thm1.4} If $k\in\N$, $r\in\N_0$, $1\leq p\leq \infty$ and $f\in\B_p^r$,  then, for all $0<t\leq 2/k$,
\be   \label{maintheorem}
c  K^\varphi_{k,r}(f^{(r)},t^k)_p \leq \omega_{k,r}^{*\varphi}(f^{(r)},t)_p \leq \wkr(f^{(r)},t)_p\leq c  K^\varphi_{k,r}(f^{(r)},t^k)_p   ,
\ee
where constants $c$ may depend only on $k$, $r$ and $p$.
\end{theorem}

\begin{remark}
Note that with an additional restriction that $t\leq t_0$, in the case $r=0$,  Theorem~{\rm\ref{thm1.4}}  becomes \cite[Theorem 2.1.1]{dt} (with $\varphi(x)= \sqrt{1-x^2}$), and that
$t_0$ can be taken to be $(2k)^{-1}$ as was shown in \cite[Theorem 6.6.2]{dl}.
\end{remark}

\begin{remark} \label{remequivdt}
It follows from $\thm{thm1.4}$ that
\[
\wkr(f^{(r)},t)_p \sim \omega_{\varphi}^k(f^{(r)},t)_{\varphi^r,p} ,
\]
where $\omega_{\varphi}^k(f^{(r)},t)_{\varphi^r,p}$ is a weighted DT modulus defined in $\ineq{dtweighted}$ (see Section~$\ref{sec3}$ for more details).
\end{remark}

Since it is obvious that
\[
\omega_{k,r}^{*\varphi}(f^{(r)},t)_p\leq\omega_{k,r}^{\varphi}(f^{(r)},t)_p,
\]
only the first and last inequalities in \ineq{maintheorem} need to be proved. Their proofs are given, respectively, in
Sections~\ref{lowerest} and \ref{upperest}.

We conclude this section with an immediate consequence of Theorem \ref{thm1.4}.
\begin{corollary} \label{maincorol}
 Let $k\in\N$, $r\in\N_0$, $1\leq p\leq \infty$, $f\in\B_p^r$ and
  $\lambda \geq 1$. Then, for all $t>0$,
$$
\wkr(f^{(r)},\lambda t)_p\le c \lambda^k\wkr(f^{(r)},t)_p .
$$
\end{corollary}

\begin{proof}
Using \thm{thm1.4}, identity \ineq{larget} and the monotonicity, in $t$, of both the $K$-functional $K^\varphi_{k,r}(f^{(r)},t^k)_p$ and the modulus $\wkr(f^{(r)},  t)_p$, and denoting
$\tilt := \min\{\lambda t, 2/k \}$,  we have
 \begin{eqnarray*}
\wkr(f^{(r)},\lambda t)_p  & = & \wkr(f^{(r)},\tilt )_p \leq c K^\varphi_{k,r}(f^{(r)},\tilt^k)_p \leq c \lambda^k K^\varphi_{k,r}(f^{(r)},(\tilt/\lambda)^k)_p \\
&\leq & c \lambda^k \wkr(f^{(r)}, \tilt/\lambda)_p \leq c \lambda^k \wkr(f^{(r)}, t)_p ,
\end{eqnarray*}
for any $t>0$.
\end{proof}

 \begin{remark}
With an additional restriction that $t\leq t_0$, in the case $r=0$,   $\cor{maincorol}$ becomes \cite[Theorem 4.1.2]{dt} .
\end{remark}

\sect{Auxiliary results}\label{sec2}

In the following proposition, we list several useful properties of the weights $\wt_\delta(x)$ and the sets $\Dom_\delta$, $\delta >0$, which will be  used below.
(Note that the statements in this proposition are vacuously true for $\delta$'s such that $\Dom_\delta$ and/or $\Dom_{2\delta}$ are empty.)

\begin{proposition}[Properties of $\wt_\delta(x)$ and $\Dom_\delta$] \label{prop}\mbox{}

\begin{enumerate}[\bf (i)]
\item \label{useful}
$\wt_\delta(x)\le \varphi(u)$, for   $x\in\Dom_\delta$  and
$u\in\left[-|x|-\delta\varphi(x)/2,|x|+\delta\varphi(x)/2\right]$.

 \item \label{compare}
 $\wt_\delta(x)\leq \varphi(x)$, for
$x\in\Dom_\delta$.

\item \label{compar}
$\varphi(x)\leq 2\wt_\delta(x)$, for $x\in\Dom_{2\delta}$.

\item \label{auxineq}
$ \delta|\varphi'(x)|\le 1$, for $x\in\Dom_{\delta}$.

\item \label{3.1}
If $y(x):=x+\delta_1\varphi(x)/2$ and $|\delta_1|\leq \delta$, then
$1/2 \leq y'(x)\leq 3/2$, for all  $x\in\Dom_{\delta}$.

\item \label{lastprop}
If $\delta_1> \delta_2$, then $\Dom_{\delta_1}\subset\Dom_{\delta_2}$.
\end{enumerate}
\end{proposition}

\begin{proof}
For $x\in\Dom_\delta$ and $u\in\left[-|x|-\delta\varphi(x)/2,|x|+\delta\varphi(x)/2\right]$, we have
$$
\varphi^2(u)-\wt_\delta^2(x)\ge\varphi^2(|x|+\delta\varphi(x)/2)-\wt_\delta^2(x)
=(1-|x|-\delta\varphi(x)/2)\delta\varphi(x)\ge0,
$$
which implies \ineq{useful}.

Choosing $u$ to be $x$ in \ineq{useful} we get \ineq{compare}.

Now, $\delta \leq 1$    implies $1+|x|\le
2(1+|x|-\delta\varphi(x)/2)$, and $x\in\Dom_{2\delta}$, that is,
$\delta\varphi(x)\le 1-|x|$, yields $1-|x|\le
2(1-|x|-\delta\varphi(x)/2)$. Hence,
$$
\varphi^2(x)=(1-|x|)(1+|x|)\le
4(1-|x|-\delta\varphi(x)/2)(1+|x|-\delta\varphi(x)/2)=4
\wt_\delta^2(x),\quad x\in\Dom_{2\delta},
$$
and so \ineq{compar} is verified. If $\delta > 1$ then \ineq{compar}  is vacuously true since $\Dom_{2\delta}=\emptyset$.

If $x\in\Dom_{\delta}$, then  $\delta\varphi(x)/2\le
1-|x|=(1+|x|)^{-1}\varphi^2(x)$, that is $\varphi(x)\ge
\delta (1+|x|)/2$. Hence
$$
\delta|\varphi'(x)|=\delta\frac{|x|}{\varphi(x)}\le\frac{2|x|}{1+|x|}\le1,
$$
which is \ineq{auxineq}.

Property \ineq{3.1} immediately follows from \ineq{auxineq}, and \ineq{lastprop} is obvious.
\end{proof}

The first important property of the new moduli is stated in the
following lemma.

\begin{lemma}\label{lem1}
If $r\in\N_0$, $1\leq p < \infty$ and
 $f\in\B^r_p$, then
$$
\lim_{t\to0^+} \omega_{k,r}^\varphi(f^{(r)},t)_p = 0 .
$$
\end{lemma}
\begin{proof} Let $\epsilon>0$. Then there is $\delta>0$  such that
$$
\int_{[-1,1]\setminus
\Dom_{\delta}}|\varphi^r(x)f^{(r)}(x)|^pdx<\left(\frac\epsilon{2^{k+2}}\right)^p.
$$
Set
$$
g^{(r)}(x):=\begin{cases} f^{(r)}(x),&\quad\text{if}\quad
x\in\Dom_{\delta},\\
0,&\quad\text{otherwise.}\end{cases}
$$
Since $g^{(r)}\in L_p[-1,1]$,
 there exists $t_0>0$ such that
$$
\omega_k^\varphi(g^{(r)},t)_p< \epsilon/2,\quad 0<t\le t_0.
$$
For each $h>0$, we have
\begin{eqnarray*}
\lefteqn{ \norm{ \wt_{kh}^r(\cdot) \Delta_{h\varphi(\cdot)}^k
(f^{(r)},\cdot)}{p} }\\
&\le& \norm{ \wt^r_{kh}(\cdot)
\Delta_{h\varphi(\cdot)}^k (g^{(r
)},\cdot)}{p}+\norm{\wt_{kh}^r(\cdot)
\Delta_{h\varphi(\cdot)}^k (f^{(r )}-g^{(r )},\cdot)}{p}\\
&\le & \norm{ \Delta_{h\varphi(\cdot)}^k (g^{(r )},\cdot)}{p}
+\norm{ \wt_{kh}^r(\cdot) \Delta_{h\varphi(\cdot)}^k (f^{(r )}-g^{(r
)},\cdot)}{p}\\&=:&I_1+I_2.
\end{eqnarray*}
Now if $h\le t_0$, then
$I_1< \epsilon/2$,
and
\begin{eqnarray*}
I_2&\le&
\sum_{i=0}^k{\binom ki}\left(\int_{\Dom_{kh}}
\left(\wt_{kh}^r(x)|f^{(r)}(x+(i-k/2)h\varphi(x)) \right.\right.\\
&&\left. \left.\qquad\qquad
-g^{(r)} (x+(i-k/2)h\varphi(x))|\right)^pdx\right)^{1/p}\\
&\le&
\sum_{i=0}^k{\binom ki}\biggl(\int_{\Dom_{kh}}
\left(\varphi^r(x+(i-k/2)h\varphi(x))|f^{(r)}(x+(i-k/2)h\varphi(x))\right.
\\
&&\left.\qquad\qquad
-g^{(r)}(x+(i-k/2)h\varphi(x))|\right)^pdx\biggr)^{1/p}\\&\le&2\sum_{i=0}^k{\binom ki}\left(\int_{-1}^1
\left(\varphi^r(u)|f^{(r)}(u)-g^{(r)}(u)|\right)^pdu\right)^{1/p}
\\&\le&2\sum_{i=0}^k{\binom ki}\left(\int_{[-1,1]\setminus\Dom_\delta}
|\varphi^r(u)f^{(r)}(u)|^p\right)^{1/p}\le\epsilon/2,
\end{eqnarray*}
where for the second inequality we used, for $x\in\Dom_{kh}$, the
inequality $\wt_{kh}(x)\le\varphi(u(x))$, where
$u=u(x):=x+(i-k/2)h\varphi(x))$, and the third inequality follows
because \prop{prop}{3.1} implies that $u'(x) \geq 1/2$ when $x\in \Dom_{kh}$. This completes the
proof.
\end{proof}

\begin{remark}
Note that $f\in\B^r_\infty$ only implies that $\omega_{k,r}^\varphi(f^{(r)},t)_\infty < \infty$ for $t>0$ and does NOT imply that $\lim_{t\to0^+} \omega_{k,r}^\varphi(f^{(r)},t)_\infty = 0$ even if we assume that $f^{(r)}\in\C(-1,1)$. For example, if $f$ is such that $f^{(r)}(x) := \varphi^{-r}(x)$, $r\in\N$, then
$f\in\B^r_\infty \cap\C^r(-1,1)$ and $\omega_{k,r}^\varphi(f^{(r)},t)_\infty \geq \const > 0$.

In was proved in \cite{sh} (see also \cite{sh-book})  that, for $r\in\N$ and $f\in\C^r(-1,1)$,
\[
\lim_{t\to0^+} \omega_{k,r}^\varphi(f^{(r)},t)_\infty = 0 \quad \mbox{if and only if} \quad  \lim_{x \to\pm 1}  \varphi^r(x) f^{(r)}(x) = 0 .
\]
In the case $r=0$, it is   easy to see $($see also \cite[p. 37]{dt}$)$ that $\lim_{t\to0^+} \omega_{k,0}^\varphi(f ,t)_\infty = 0$ if and only if $f\in\C[-1,1]$.
\end{remark}

\begin{lemma} \label{lemsec2} Let $1\leq p\leq \infty$ and $r\in\N_0$. If $g\in\B^{r+1}_p$, then
\[
\norm{ \varphi^\gamma g^{(r)}}{p}<\infty ,
\]
 for any $\gamma \geq 0$ such that $\gamma > r-1$.
  \end{lemma}

\begin{corollary} \label{corsec2} Let $1\leq p\leq \infty$ and $r\in\N_0$. If $g\in\B^{r+1}_p$, then $g\in\B^{r}_p$.
\end{corollary}

 \begin{proof}[Proof of $\lem{lemsec2}$]
Suppose that we are given $g\in\B^{r+1}_p$.
Without loss of generality we can assume that $g^{(r)}(0)=0$.

First, we consider the case $p=\infty$. Since $\varphi(u)\ge\varphi(x)$
for $|u|\le|x|$, we have
\begin{eqnarray*}
\norm{ \varphi^\gamma g^{(r)}}{\infty}  & \leq &   \norm{\varphi^{\gamma}(x)  \int_0^x    g^{(r+1)}(u) \, du}{\infty}\\
& \leq & \norm{ \varphi^{r+1}  g^{(r+1)}}{\infty}
\norm{ \varphi^{\gamma}(x)    \int_0^x  \varphi^{-r-1}(u) \, du}{\infty}   \\
& \leq & \norm{ \varphi^{r+1}  g^{(r+1)}}{\infty}
\norm{    \int_0^x  \varphi^{\gamma-r-1}(u) \, du}{\infty}   \\
& \leq & c \norm{ \varphi^{r+1}  g^{(r+1)}}{\infty}  .
\end{eqnarray*}

Similarly, if  $p=1$, then
\begin{eqnarray*}
\norm{\varphi^\gamma g^{(r)}}{1}&=&\int_{-1}^1 \varphi^{\gamma}(x)
\left|\int_0^xg^{(r+1)}(u)\,du\right|\,dx\\
&\le&\int_{-1}^1  \varphi^{\gamma-r-1}(x) \left|\int_0^x
\varphi^{r+1}(u)|g^{(r+1)}(u)|\,du\right|\,dx\\&\leq&\norm{
\varphi^{r+1}g^{(r+1)}}{1}\int_{-1}^1 \varphi^{\gamma-r-1}(x) dx   \leq c\norm{
\varphi^{r+1} g^{(r+1)}}{1}.
\end{eqnarray*}

Suppose now that $1<p <\infty$ and $q = p/(p-1)$ is such that $(r+1)q\ne 2$ (\ie either $r\ge1$, or $r=0$ and $p\ne2$).
Using H\"older inequality we have
\begin{eqnarray*}
\lefteqn{ \norm{\varphi^\gamma g^{(r)}}{p}^p }\\
&=& \int_{-1}^1\varphi^{\gamma p}(x)\left|\int_0^xg^{(r+1)}(u)\,du\right|^p\,dx \\
& \leq& \int_{-1}^1\varphi^{\gamma p}(x)\left|\left(\int_0^x
\varphi^{-(r+1)q}(u)\, du\right)^{1/q}
\left(\int_0^x|\varphi^{r+1}(u)g^{(r+1)}(u)|^p\,du\right)^{1/p}\right|^p\,dx\\
& \leq&\norm{\varphi^{r+1} g^{(r+1)}}{p}^p
\int_{-1}^1\varphi^{\gamma p}(x)\left| \int_0^x\varphi^{-(r+1)q}(u)\, du\right|^{p/q}\,dx\\
& =&2 \norm{ \varphi^{r+1} g^{(r+1)}}{p}^p
\int_{0}^1\varphi^{\gamma p}(x)\left(\int_0^x\varphi^{-(r+1)q}(u)\, du\right)^{p/q}\,dx\\
& \leq&2^{1+\gamma p/2} \norm{\varphi^{r+1} g^{(r+1)}}{p}^p
\int_{0}^1(1-x)^{\gamma p/2}\left(\int_0^x(1-u)^{-(r+1)q/2}\,du\right)^{p/q}\,dx\\
& \leq& c \norm{\varphi^{r+1} g^{(r+1)}}{p}^p
\int_{0}^1(1-x)^{\gamma p/2} \max\{1, (1-x)^{-1-p(r-1)/2} \} dx \\
& \leq& c \norm{\varphi^{r+1} g^{(r+1)}}{p}^p
\int_{0}^1 \max\{(1-x)^{\gamma p/2}, (1-x)^{-1+p(\gamma -r+1)/2} \} dx \\
& \leq &c\norm{\varphi^{r+1}g^{(r+1)}}{p}^p .
\end{eqnarray*}
Finally, if $p=2$ and $r=0$, then
\begin{eqnarray*}
\norm{\varphi^\gamma g}{2}^2 &\leq&
2^{1+\gamma} \norm{\varphi g'}{2}^2\int_{0}^1 (1-x)^\gamma \int_0^x\frac{du}{(1-u)}\,dx \\
&\leq&
c \norm{\varphi
g'}{2}^2\int_{0}^1\int_0^x\frac{du}{(1-u)}\,dx \leq c \norm{\varphi
g'}{2}^2.
\end{eqnarray*}
\end{proof}

\sect{Proof of \thm{thm1.4}: the upper estimate}\label{upperest}

The upper estimate of our modulus by the $K$-functional in \thm{thm1.4} (\ie the last inequality in \ineq{maintheorem}) immediately follows from the following lemma.

\begin{lemma}\label{less}
If $k\in\N$, $r\in\N_0$, $1\leq p\leq \infty$ and $f\in\B_p^r$,  then
$$
\wkr(f^{(r)},t)_p\le c(k,r,p)  K^\varphi_{k,r}(f^{(r)},t^k)_p, \quad \mbox{\rm for all } t >0 .
$$
\end{lemma}

\begin{proof}
In view of \ineq{larget} and the monotonicity of the $K$-functional in $t$, we may assume that $t\leq 2/k$.
Take any $g\in\B_p^{k+r}$.
\cor{corsec2} implies that $g\in\B_p^r$, whence
\[
\wkr(f^{(r)},t)_p\leq\wkr(f^{(r)}-g^{(r)},t)_p+\wkr(g^{(r)},t)_p .
\]
Take $h$ such that $0<h \leq  t$.

For each  $0\leq i\leq k$,   put
$y_i(x):=x+(i-k/2)h\varphi(x)$.
Then \prop{prop}{3.1} implies that $y_i'(x) \geq 1/2$, for $x\in\Dom_{kh}$, and so we have
(with obvious modifications if $p=\infty$)
\begin{eqnarray*}
\lefteqn{ \|\varphi^r(y_i)(f^{(r)}(y_i)-g^{(r)}(y_i)))\|_{L_p(\Dom_{kh})} }\\
&=& \left( \int_{\Dom_{kh}}\varphi^{rp}(y_i(x))|f^{(r)}(y_i(x))-g^{(r)}(y_i(x))|^p\,dx \right)^{1/p}\\
&\le& 2^{1/p} \left(\int_{-1}^1\varphi^{rp}(y)|f^{(r)}(y)-g^{(r)}(y)|^p\,dy\right)^{1/p} \\
&=&2^{1/p} \|\varphi^r(f^{(r)}-g^{(r)})\|_p .
\end{eqnarray*}

Since $\wt_\delta(x)\leq\varphi(y)$ for all $x\in\Dom_\delta$,
$y\in \left[x-\delta\varphi(x)/2,x+\delta\varphi(x)/2\right]$ and
$0<\delta \leq 2$, we get
\begin{eqnarray*}
 \wkr(f^{(r)}-g^{(r)},t)_p
&\leq&\sup_{0< h \leq t} \norm{\sum_{i=0}^k {\binom ki}
\varphi^r(y_i)\left|f^{(r)}(y_i)-g^{(r)}(y_i)\right|}{\Lp(\Dom_{kh})} \\
&\leq& 2^{k+1/p}\norm{\varphi^r(f^{(r)}-g^{(r)})}{p}.
\end{eqnarray*}
To estimate the second term $\wkr(g^{(r)},t)_p $,  using the identity
\be \label{diffid}
\Delta_h^k(f,x) = \int_{-h/2}^{h/2}\cdots
\int_{-h/2}^{h/2}f^{(k)}(x+u_1+\cdots+u_k)du_1\cdots du_k ,
\ee
we have
\begin{eqnarray*}
\lefteqn{ \wkr(g^{(r)},t)_p }\\
&=&\sup_{0<h\leq t}\norm{
\wt_{kh}^r\Delta_{h\varphi}^k(g^{(r)},\cdot)}{\Lp(\Dom_{kh})}  \\
&=&\sup_{0<h\leq t}\norm{ \wt_{kh}^r
\int_{-h\varphi/2}^{h\varphi/2}\cdots\int_{-h\varphi/2}^{h\varphi/2}
g^{(k+r)}(\cdot+u_1+\cdots+u_k)du_1\cdots du_k}{\Lp(\Dom_{kh})} .
\end{eqnarray*}
By H\"older's inequality (with $1/p + 1/q=1$), for each $u$
satisfying $-1<x+u-h\varphi(x)/2<x+u+h\varphi(x)/2<1$, we have
\begin{eqnarray*}
\left|\int_{-h\varphi(x)/2}^{h\varphi(x)/2} g^{(k+r)}(x
+u+u_k)du_k\right|&=&\left|\int_{x+u-h\varphi(x)/2}^{x+u+h\varphi(x)/2}
g^{(k+r)}(v)dv\right|\\
&\le&\int_{x+u-h\varphi(x)/2}^{x+u+h\varphi(x)/2}
\frac{\varphi^{k+r}(v)|g^{(k+r)}(v)|}{\varphi^{k+r}(v)}dv\\
&\le&\|\varphi^{k+r}g^{(k+r)}\|_{L_p(\mathcal{A}(x,u))}\|\varphi^{-k-r}\|_{L_q(\mathcal{A}(x,u))} ,
\end{eqnarray*}
where
$$
\mathcal{A}(x,u):=\left[x+u-\frac h2\varphi(x), x+u+\frac
h2\varphi(x)\right].
$$
Thus, in order to complete the proof, it suffices to prove
\begin{eqnarray} \label{kge1}
&& \int_{\Dom_{kh}}\biggl(\wt_{kh}^r(x)\int_{-h\varphi(x)/2}^{h\varphi(x)/2}
\cdots\int_{-h\varphi(x)/2}^{h\varphi(x)/2}\|\varphi^{-k-r}\|_{\Lq(\mathcal{A}(x,u_1+\cdots+u_{k-1}))}
\\ \nonumber
&&\qquad\times
\|\varphi^{k+r}g^{(k+r)}\|_{L_p(\mathcal{A}(x,u_1+\cdots+u_{k-1}))}du_1\cdots
du_{k-1}\biggr)^pdx\\ \nonumber
&&\le ch^{kp}\|g^{(k+r)}\varphi^{k+r}\|_p^p ,
\end{eqnarray}
noting that, in the case $k=1$, this inequality is understood as
\begin{eqnarray} \label{keq1}
&& \int_{\Dom_{h}}\biggl(\wt_{h}^r(x)
 \|\varphi^{-1-r}\|_{\Lq(\mathcal{A}(x,0))}
\|\varphi^{1+r}g^{(1+r)}\|_{L_p(\mathcal{A}(x,0))}
 \biggr)^pdx\\ \nonumber
&&\le ch^{p}\|g^{(1+r)}\varphi^{1+r}\|_p^p ,
\end{eqnarray}
 and, if $p=\infty$, then \ineq{kge1} is replaced by
\begin{eqnarray} \label{keqinf}
 && \sup_{x\in \Dom_{kh}}\biggl(\wt_{kh}^r(x)\int_{-h\varphi(x)/2}^{h\varphi(x)/2}
\cdots\int_{-h\varphi(x)/2}^{h\varphi(x)/2}\|\varphi^{-k-r}\|_{\L_1(\mathcal{A}(x,u_1+\cdots+u_{k-1}))}
\\ \nonumber
&&\qquad\times
\|\varphi^{k+r}g^{(k+r)}\|_{L_\infty (\mathcal{A}(x,u_1+\cdots+u_{k-1}))}du_1\cdots
du_{k-1}\biggr)  \\ \nonumber
&&\le ch^{k }\|g^{(k+r)}\varphi^{k+r}\|_\infty  .
\end{eqnarray}
To this end we write,
$$
\int_{\Dom_{kh}}=\int_{\Dom_{2kh}}+\int_{(\Dom_{kh}\setminus
\Dom_{2kh})\cap[0,1]}+\int_{(\Dom_{kh}\setminus
\Dom_{2kh})\cap[-1,0]}=:I_1(p)+I_2(p)+I_3(p)
$$
if $1\leq p<\infty$, and
\[
\sup_{\Dom_{kh}} \leq \sup_{\Dom_{2kh}}+\sup_{(\Dom_{kh}\setminus
\Dom_{2kh})\cap[0,1]}+\sup_{(\Dom_{kh}\setminus
\Dom_{2kh})\cap[-1,0]}=:I_1(\infty)+I_2(\infty)+I_3(\infty)
\]
if $p=\infty$.

{\bf \vspace{0.5em} \noindent Part I: estimate of $I_1$} \vspace{0.5em}

First, we note that, if $h>1/k$, then  $\Dom_{2kh}=\emptyset$, and so no estimate of $I_1$ is needed.
Hence, in this part, we may assume that $h\leq 1/k$.

For $x\in\Dom_{2\delta}$ and
$u\in[x-\delta\varphi(x)/2,x+\delta\varphi(x)/2]$, we have
$$
\frac12(1-|u|)\le\frac12(1-|x|+\delta\varphi(x)/2)\le1-|x|\le2(1-|x|-\delta\varphi(x)/2)\le2(1-|u|),
$$
where for the second and third inequalities we applied the fact that
$\delta\varphi(x)\le1-|x|$. Also, obviously,
$$
\frac12(1+|u|)\le1+|x|\le2(1+|u|).
$$
Hence, for $x\in\Dom_{2\delta}$ and
$u\in[x-\delta\varphi(x)/2,x+\delta\varphi(x)/2]$,
\be \label{phi}
\frac12\varphi(u)\le\varphi(x)\le2\varphi(u).
\ee
Also, note  that $\delta\varphi(x)\le1-|x|$ (\ie $x\in \Dom_{2\delta}$) implies
\be \label{delta}
\delta\le\varphi(x).
\ee
So, if $x\in\Dom_{2kh}$, then by\ineq{phi},
\begin{align*}
&\wt_{kh}^r(x)\int_{-h\varphi(x)/2}^{h\varphi(x)/2}
\cdots\int_{-h\varphi(x)/2}^{h\varphi(x)/2}\|\varphi^{-k-r}\|_
{\Lq(\mathcal{A}(x,u_1+\dots+u_{k-1}))}\,du_1\cdots\,du_{k-1}
\\&
\le\varphi^{r}(x)(h\varphi(x))^{k-1}\frac{2^{k+r}}{\varphi^{k+r}(x)}(h\varphi(x))^{1/q}
\\&=2^{k+r}h^{k-1+1/q}\varphi^{1/q-1}(x)=2^{k+r}h^{k-1/p}\varphi^{-1/p}(x),
\end{align*}
where we applied \prop{prop}{compare}. Note that the above estimate
is also valid 
for $k=1$.

Therefore by \eqref{phi} and \eqref{delta}, for $1\leq p< \infty$ we have
\begin{eqnarray*}
I_1(p) &\le&\int_{\Dom_{2kh}}\biggl(2^{k+r}h^{k-1/p}\varphi^{-1/p}(x)
\|\varphi^{k+r}g^{(k+r)}\|_{L_p([x-kh\varphi(x)/2,x+kh\varphi(x)/2])}\biggr)^pdx\\
&=&
2^{p(k+r)}h^{kp-1}\int_{\Dom_{2kh}}\frac1{\varphi(x)}\int_{x-kh\varphi(x)/2}^{x+kh
\varphi(x)/2}\left|\varphi^{k+r}(u)g^{(k+r)}(u)\right|^pdudx\\
&\le&ch^{kp-1}\int_{\Dom_{2kh}}\int_{x-kh\varphi(x)/2}^{x+kh\varphi(x)/2}\frac1{\varphi(u)+kh/2}
\left|\varphi^{k+r}(u)g^{(k+r)}(u)\right|^pdudx\\
&=&ch^{kp-1}\int_{a_1}^{a_2}\int_{b_1(u)}^{b_2(u)}\frac1{\varphi(u)+kh/2}
\left|\varphi^{k+r}(u)g^{(k+r)}(u)\right|^pdxdu\\
&=&ch^{kp-1}\int_{a_1}^{a_2}\left|\varphi^{k+r}(u)g^{(k+r)}(u)\right|^p\frac{b_2(u)-b_1(u)}{\varphi(u)+kh/2}du,
\end{eqnarray*}
where $-1<a_1<a_2<1$ and
$$
b_2(u)-b_1(u)\le \frac{kh\sqrt{1-u^2+(kh/2)^2}}{1+(kh/2)^2}\le kh(\varphi(u)+ kh/2 ).
$$
Hence,
$$
I_1(p)\le ch^{kp}\|g^{(k+r)}\varphi^{k+r}\|_p^p.
$$
If $p=\infty$, then
\begin{eqnarray*}
I_1(\infty) &\le&\sup_{\Dom_{2kh}}\biggl(2^{k+r}h^{k}
\|\varphi^{k+r}g^{(k+r)}\|_{L_\infty([x-kh\varphi(x)/2,x+kh\varphi(x)/2])}\biggr) \\
&\leq & c h^k \|\varphi^{k+r}g^{(k+r)}\|_{\infty }  .
\end{eqnarray*}

 {\bf \vspace{0.5em} \noindent Part II: estimate of $I_2$} \vspace{0.5em}

In this part, we estimate $I_2$, the estimate of $I_3$ being completely analogous.
It is convenient to introduce the notation

\[
 \F_q(x,k,r):=\wt_{kh}^r(x)\int_{-h\varphi(x)/2}^{h\varphi(x)/2}
\cdots\int_{-h\varphi(x)/2}^{h\varphi(x)/2}
\|\varphi^{-k-r}\|_{\Lq(\mathcal{A}(x,u_1+\cdots+u_{k-1}))}
du_1\cdots du_{k-1},
\]
\[
\F_q(x,k):=\F_q(x,k,0),
\]
and
\[
\E_{kh} := (\Dom_{kh}\setminus \Dom_{2kh})\cap[0,1].
\]

The required estimates for $I_2(p)$ and $I_2(\infty)$ follow, respectively, from
\be\label{kge1a}
\int_{\E_{kh}}\left(\F_q(x, k, r)\right)^pdx
\le ch^{kp},
\ee
and
\be\label{keqinfa}
\sup_{x\in\E_{kh}}\F_1(x, k, r)\le ch^k.
\ee

First we observe that if $v\in \mathcal{A}(x,u_1+\cdots+u_{k-1})$ and $|u_i|\le h\varphi(x)/2$, then
\[
x-kh\varphi(x)/2\le v\le x+kh\varphi(x)/2,
\]
which, by \prop{prop}{useful} implies, for $x\in\Dom_{kh}$, that
\[
\wt_{kh}(x)\le\varphi(v).
\]
Hence
\[
\wt_{kh}^r(x)\|\varphi^{-k-r}\|_{\Lq(\mathcal{A}(x,u_1+\cdots+u_{k-1}))}
\le \|\varphi^{-k}\|_{\Lq(\mathcal{A}(x,u_1+\cdots+u_{k-1}))},
\]
so that
\[
\F_q(x,k,r)\le\F_q(x,k),\quad x\in\Dom_{kh}.
\]
Thus, \ineq{kge1a} and \ineq{keqinfa} follow, respectively, from
\be\label{kge1ab}
\int_{\E_{kh}}\left(\F_q(x,k)\right)^pdx
\le ch^{kp},
\ee
and
\be  \label{keqinfab}
\sup_{x\in\E_{kh}}\F_1(x,k)\le ch^k.
\ee

 Recall that $\mu(\delta) = 2\delta^2/(4+\delta^2)$ and note that
$\E_{kh}=(1-\mu(2kh),1-\mu(kh)]\cap[0,1]$, \ie
\[
\E_{kh}=
\begin{cases}
(1-\mu(2kh),1-\mu(kh)] , & \mbox{\rm if}\quad   h\leq 1/k , \\
 [0,1-\mu(kh)], & \mbox{\rm if}\quad   1/k<h\leq2/k .
 \end{cases}
 \]

It will be convenient for us to separate the proof   for ``small'' and ``large''  $h$.
We first consider the case when $h\leq 1/(\sqrt{2} k)$.

{\bf \vspace{0.5em} \noindent Part II(i): $h\leq 1/(\sqrt{2} k)$} \vspace{0.5em}

It is easy to see that, if $h\leq 1/(\sqrt{2} k)$, then
\[
{8 k^2h^2 \over 9} \leq \meas(\E_{kh}) \leq  {3  k^2h^2 \over 2}
\]
and, for $x\in \E_{kh}$,
 \[
 {4 k^2h^2 \over 9} \leq 1-x \leq 2k^2h^2 \andd  { 2kh \over 3} \leq\varphi(x)\leq2kh .
 \]
It is important to note that, if $h\leq 1/(\sqrt{2} k)$ and $x\in \E_{kh}$, then
\be \label{smallh}
x - kh\varphi(x)/2 \geq 0 .
\ee
This implies that, if  $v\in \A(x, u_1+\dots u_{k-1})$ where $x\in \E_{kh}$ and
  $|u_i|\le h\varphi(x)/2$,  then $v\geq 0$ and so
\[
\sqrt{1-v} \leq \varphi(v) \leq \sqrt{2(1-v)} .
\]

Now, for any $q<\infty$,   $x\in \E_{kh}$ and $u\in [-(k-1)h\varphi(x)/2, (k-1)h\varphi(x)/2]$,  we have
\begin{eqnarray} \label{eqnarp}
 \lefteqn{  \|\varphi^{-k }\|_{\L_q(\mathcal{A}(x,u))}  } \\ \nonumber
&\le&
\left(\int_{x+u -
h \varphi(x)/2}^{x+u+
h \varphi(x)/2} (1-v)^{- kq/2} dv\right)^{1/q}\\ \nonumber
& \le  & c(k,q)
\begin{cases}
\left(1-x  -u -  h \varphi(x)/2\right)^{-k/2+1/q} , & \mbox{\rm if } kq > 2 , \\
\left(  \ln {1-x -u  +  h \varphi(x)/2 \over  1-x -u -  h \varphi(x)/2 }  \right)^{1/q} , & \mbox{\rm if } kq = 2 ,\\
h^{-k +2/q} , & \mbox{\rm if } kq < 2 ,
\end{cases}
\end{eqnarray}
 and note that \ineq{eqnarp} is also valid if $q=\infty$.

We also observe that, if $k\geq 2$, then for any  $x\in \E_{kh}$ and $u\in [-(k-2)h\varphi(x)/2, (k-2)h\varphi(x)/2]$,
\be \label{ineqln}
\int_{-h\varphi(x)/2}^{h\varphi(x)/2}  \int_{-h\varphi(x)/2}^{h\varphi(x)/2} (1-x-u-u_1-u_2)^{-1} du_1 du_2 \leq c h^2 .
\ee
Indeed, using the fact that $1-x-u\geq h\varphi(x)$ and changing variables to $v:= - 2u_1/(h\varphi(x))$ and $w:= - 2u_2/(h\varphi(x))$,  we have
\begin{eqnarray*}
\lefteqn{ \int_{-h\varphi(x)/2}^{h\varphi(x)/2}  \int_{-h\varphi(x)/2}^{h\varphi(x)/2} (1-x-u-u_1-u_2)^{-1} du_1 du_2 }\\
& \leq& \int_{-h\varphi(x)/2}^{h\varphi(x)/2}  \int_{-h\varphi(x)/2}^{h\varphi(x)/2} (h\varphi(x)-u_1-u_2)^{-1} du_1 du_2 \\
& = & \frac{h\varphi(x)}{2}  \int_{-1}^{1}  \int_{-1}^{1} (2+v+w)^{-1} dv dw \\
&= & (2\ln 2) h \varphi(x) \leq c(k)  h^2.
\end{eqnarray*}
Thus, \ineq{ineqln} implies that, for any $k\geq 2$, $\alpha > 1-k$, $x\in \E_{kh}$ and $u\in [- h\varphi(x)/2,  h\varphi(x)/2]$,
\begin{eqnarray}\label{auxln}
\int_{-h\varphi(x)/2}^{h\varphi(x)/2} \cdots  \int_{-h\varphi(x)/2}^{h\varphi(x)/2} &(1-x-u-u_1-\cdots -u_{k-1})^{\alpha}du_1\dots du_{k-1}
\\&\nonumber\leq ch^{2k+2\alpha-2} .
\end{eqnarray}

Now, for any  $x\in \E_{kh}$,  we have
\begin{eqnarray*}
  \F_1(x,k)
  & = & \int_{-h\varphi(x)/2}^{h\varphi(x)/2} \cdots\int_{-h\varphi(x)/2}^{h\varphi(x)/2}   \varphi^{-k}(x+u_1+\dots + u_k)   du_1\cdots du_{k}\\
 & \le &
 \int_{-h\varphi(x)/2}^{h\varphi(x)/2} \cdots\int_{-h\varphi(x)/2}^{h\varphi(x)/2}     (1-x-u_1-\cdots-u_{k})^{-k/2}  du_1\cdots du_{k}\\
 & \le & c h^k ,
\end{eqnarray*}
which implies \ineq{keqinfab} and so completes the proof in the case $p=\infty$.

For $1\leq p<\infty$, 
it is convenient to break the proof of \ineq{kge1ab} into several cases.

\case{$1\leq p<\infty$,  $k=1$ and $q>2$}

Using \ineq{eqnarp} we have,

\begin{align*}
\int_{\E_{h}}\left(\F_q(x,1)\right)^pdx&=\int_{\E_{h}}
\|\varphi^{-1}\|_{\L_q(\mathcal{A}(x,0))}^p dx\\
&\le c\int_{\E_{h}} (1-x-h\varphi(x)/2)^{-p/2+p/q} dx\\
& =   c \int_{1-\mu(2h)}^{1-\mu(h)}  (1-x-h\varphi(x)/2)^{p/2 -1 } dx \\
& \le
c  \int_{1/(1+h^2)}^{1}
  (1-y)^{p/2 -1 } dy \\
  & \le   c h^p   .
\end{align*}

\case{$1\leq p<\infty$, $k=1$ and  $q=2$}

Applying \ineq{eqnarp} we obtain,
\begin{eqnarray*}
\int_{\E_{h}}\left(\F_2(x,1)\right)^2 dx
&\le&c\int_{1-\mu(2h)}^{1-\mu(h)}\ln\left(1+\frac{h\varphi(x)}{1-x-h\varphi(x)/2}\right)\,dx\\
&\le&c\int_{1-\mu(2h)}^{1-\mu(h)}\ln\left(1+\frac{2h^2}{1-x-h\varphi(x)/2}\right)\,dx\\
&\le&c\int^1_{1/(1+h^2)}
\ln\left(1+\frac{2h^2}{1-y}\right)dy\\
&\le&ch^2  .
\end{eqnarray*}

\case{$1\leq p<\infty$, $k=1$ and  $q<2$}

We apply \ineq{eqnarp} and get
 \[
\int_{\E_{h}}\left(\F_q(x,1)\right)^pdx\le c\int_{\E_{h}}h^{p-2}dx
\le c\meas(\E_{h}) h^{p-2}\le ch^p.
\]

\case{$1\leq p<\infty$, $k\ge2$ and $2/p<k$}

Note that in this case, $kq>2$ and $k/2+1/q >1$, and
applying \ineq{eqnarp} and \ineq{auxln}, we get
\begin{eqnarray*}
\lefteqn{\int_{\E_{kh}}\left(\F_q(x,k)\right)^pdx}\\
&\le&c\int_{\E_{kh}}\biggl(\int_{-h\varphi(x)/2}^{h\varphi(x)/2}\cdots
\int_{-h\varphi(x)/2}^{h\varphi(x)/2}\\
& &(1-x-u_1-\dots-u_{k-1}-h\varphi(x)/2)^{-k/2+1/q}du_1\cdots du_{k-1}\biggr)^pdx\\
& \le & c \meas(\E_{kh}) h^{(k-2+2/q)p}\le ch^{kp}.
\end{eqnarray*}

\case{$1\leq p<\infty$ and $2/p\geq k\geq 2$}

Since $2\leq k\leq 2/p$ and $p\geq 1$ can hold simultaneously only if $k=2$ and $p=1$,  using \ineq{eqnarp} for  $q=\infty$   we have
\begin{eqnarray*}
\lefteqn{\int_{\E_{2h}}\F_\infty(x,2)dx}\\
& = & \int_{ \E_{2h} }    \int_{-h\varphi(x)/2}^{h\varphi(x)/2}
\|\varphi^{-2}\|_{\L_\infty(\mathcal{A}(x,u))} du  dx\\
&\le&
c  \int_{\E_{2h}}  \int_{-h\varphi(x)/2}^{h\varphi(x)/2}
 (1-x-u -  h\varphi(x)/2)^{-1 }du   dx \\
&=&c\int_{1-\mu(4h)}^{1-\mu(2h)}
\ln\left(1+\frac{h\varphi(x)}{1-x-h\varphi(x)}\right)dx \\
&\le&c\int_{1-\mu(4h)}^{1-\mu(2h)}
\ln\left(1 + \frac{4h^2}{1-x-h\varphi(x)}\right)dx \\
&\le&c\int_{1/(1+4h^2)}^{1}
\ln\left(1+\frac{4h^2}{1-y}\right)dy \\
&\le&ch^2 .
\end{eqnarray*}

It remains to consider the case when $1/(\sqrt{2} k) < h \leq 2/k$.

{\bf \vspace{0.5em} \noindent Part II(ii): $1/(\sqrt{2} k) < h \leq 2/k$} \vspace{0.5em}

If $1/(\sqrt{2} k) < h \leq 2/k$, then,  for $x\in \E_{kh}$, $h\sim 1-x \sim \varphi(x) \sim c(k)$, and
 $\meas(\E_{kh}) \leq c(k)$ (``$\leq$'' cannot  be replaced with ``$\sim$'' since $\meas(\E_{2})=0$).
 Inequalities \ineq{kge1ab} and \ineq{keqinfab} which we need to verify become
 \be\label{kge1abc}
\int_{\E_{kh}}\left(\F_q(x,k)\right)^pdx
\le c ,
\ee
and
\be  \label{keqinfabc}
\sup_{x\in\E_{kh}}\F_1(x,k)\le c .
\ee

We can prove \ineq{kge1abc} and \ineq{keqinfabc} using the   proof   used  in Part II(i) with the only difference that we can no longer use the fact that $\varphi(v) \sim \sqrt{1-v}$ for
$v\in \A(x, u_1+\dots + u_{k-1})$ with $x\in\E_{kh}$ and $|u_i| \leq h\varphi(x)/2$. At the same time, since we no longer need to keep track of powers of $h$'s, this proof can be considerably simplified.

First, let $F\in C[-1,1]$ be such that $F^{(k)}(x) = \varphi^{-k}(x)$, $x\in(-1,1)$. Applying the identity \ineq{diffid} for any $x\in \E_{kh}$, we have

\begin{eqnarray*}
  \F_1(x,k)
   & = &   \int_{-h\varphi(x)/2}^{h\varphi(x)/2} \cdots\int_{-h\varphi(x)/2}^{h\varphi(x)/2}    \varphi^{-k}(x+u_1+\dots + u_k)   du_1\cdots du_{k} \\
   & =&   \Delta_{h\varphi(x)}^k (F, x) \leq c ,
\end{eqnarray*}
which implies \ineq{keqinfabc} and so completes the proof in the case $p=\infty$.

Now,  observing that, for   $-1<a<b<1$,
\be\label{ineqvarphinew}
\int_a^b \varphi^{\gamma}(t) dt \leq c(\gamma)
\begin{cases}
  \varphi^{2+\gamma}(a) + \varphi^{2+\gamma}(b) , & \mbox{\rm if } \gamma <-2 ,\\
  1 , & \mbox{\rm if } \gamma > -2 ,
  \end{cases}
\ee
we conclude that
\[
\| \varphi^{-k}\|_{\L_q[a,b]}  \leq c(k,q)
\begin{cases}
  \varphi^{ -k +2/q}(a) + \varphi^{ -k +2/q}(b) , & \mbox{\rm if } kq > 2 ,\\
  1 , & \mbox{\rm if } kq <2  .
  \end{cases}
\]
Therefore, in particular,
\begin{eqnarray*}
\F_q(x,1) &=&  \| \varphi^{-1}\|_{\L_q\left[ x-h\varphi(x)/2, x+h\varphi(x)/2 \right]}  \\
&\leq&  c(k,q)
\begin{cases}
  \varphi^{ -1 +2/q}(x-h\varphi(x)/2) + \varphi^{ -1 +2/q}(x+h\varphi(x)/2) , & \mbox{\rm if }  q > 2 ,\\
  1 , & \mbox{\rm if }  q <2  .
  \end{cases}
\end{eqnarray*}
 Hence, \ineq{kge1abc} is verified if $k=1$ and $q<2$ ($p>2$), and for $q>2$ ($1\leq p <2$) we have
 \begin{eqnarray*}
\int_{\E_{h}}\left(\F_q(x,1)\right)^pdx & \leq & c \int_{\Dom_{h}} \left( \varphi^{p-2 }(x-h\varphi(x)/2) + \varphi^{p-2}(x+h\varphi(x)/2) \right) dx \\
&\leq &
c \int_{\Dom_{h}} \varphi^{p-2}(x+h\varphi(x)/2)dx \\
&\leq &
c \int_{-1}^1 \varphi^{p-2}(v)dv \leq c .
\end{eqnarray*}
Finally, if $k=1$ and $p=q=2$, then we
observe that, for a centrally symmetric set $S\subset \R^2$, we have
$\iint_S f(-\bar x) d\bar x = \iint_S f(\bar x) d\bar x$.
Hence,
 \begin{eqnarray*}
\int_{\E_{h}}\left(\F_2(x,1)\right)^2dx & = & \int_{\E_{h}} \int_{x-h\varphi(x)/2}^{x+h\varphi(x)/2} \frac{1}{1-t^2} dt dx \\
& \leq & c \int_{\Dom_{h}} \int_{x-h\varphi(x)/2}^{x+h\varphi(x)/2} \frac{1}{1-t} dt dx \\
& \leq & c - c \int_{\Dom_{h}} \ln(1-x-h\varphi(x)/2) dx \\
& \leq & c - c \int_{-1}^1 \ln(1-v) dv \leq c ,
\end{eqnarray*}
and so \ineq{kge1abc} is verified for all $1\leq p<\infty$ if $k=1$.

If $1\leq p <\infty$ and $k\geq 2$, then $kq>2$ and so
  \begin{eqnarray*}
   \F_q(x,k)
 & \leq &
 c \int_{-h\varphi(x)/2}^{h\varphi(x)/2}\cdots
\int_{-h\varphi(x)/2}^{h\varphi(x)/2} \left( \varphi^{-k+2/q}(x+u_1+\dots+u_{k-1}-h\varphi(x)/2) \right.    \\
& & \left. +
\varphi^{-k+2/q}(x+u_1+\dots+u_{k-1}+h\varphi(x)/2)  \right)   du_1\cdots du_{k-1} .
\end{eqnarray*}
Again, let $F\in C[-1,1]$, be such that $F^{(k-1)}(x) = \varphi^{-k+2/q}(x)$ (this is possible provided $k+2/q>2$). Applying the identity \ineq{diffid}, we have
  \begin{eqnarray*}
   \F_q(x,k) & \leq & c \Delta_{h\varphi(x)}^{k-1} (F, x-h\varphi(x)/2) + c \Delta_{h\varphi(x)}^{k-1} (F, x+h\varphi(x)/2) \leq c .
   \end{eqnarray*}
This implies \ineq{kge1abc} in all remaining cases except for $k=2$ and $p=1$ ($q=\infty$).
Finally, twice using the fact that $\left\{(x,u) \st x\in\Dom_{2h} , \; |u|\leq h\varphi(x)/2 \right\}$ is centrally symmetric, we have
 \begin{eqnarray*}
\lefteqn{\int_{\E_{2h}}\F_\infty(x,2)dx}\\
& \le & c \int_{ \E_{2h} } \int_{-h\varphi(x)/2}^{h\varphi(x)/2} \left( \varphi^{-2}(x+u -h\varphi(x)/2)
  +
\varphi^{-2}(x+u+h\varphi(x)/2)  \right)   du dx\\
& \leq&
c \int_{ \Dom_{2h} } \int_{-h\varphi(x)/2}^{h\varphi(x)/2} \varphi^{-2}(x+u+h\varphi(x)/2)du dx\\
&\leq&
c \int_{ \Dom_{2h} } \int_{-h\varphi(x)/2}^{h\varphi(x)/2} (1-x-u-h\varphi(x)/2)^{-1} du dx\\
&=&c\int_{\Dom_{2h}}
\left( \ln(1-x)  - \ln \bigl(1-x-h\varphi(x)\bigr) \right) dx
 \\
&\le&c\int_{-1}^{1 }|\ln(1-x)|dx \leq c .
\end{eqnarray*}
This completes the proof of the lemma.
\end{proof}
\vspace{1cm}

\sect{Weighted DT moduli}\label{sec3}

The following weighted DT moduli are defined in \cite[p.
218]{dt} (with $D=(0,1)$).
\begin{eqnarray*}
\omega_{\psi}^k(f,t)_{w,p} & := & \sup_{0<h\leq t}
\norm{w\Delta_{h\psi}^kf}{\Lp[t_0^*, 1-t_1^*]} \\
&&\ + \sup_{0< h\leq
t_0^*} \norm{w \overrightarrow\Delta_h^kf}{\Lp[0,12t_0^*]} +
\sup_{0< h\leq t_1^*}\norm{w
\overleftarrow\Delta_h^kf}{\Lp[1-12t_1^*,1]},
\end{eqnarray*}
where if $\psi(x)=\sqrt{x(1-x)}$, then $t_0^*=t_1^*=k^2t^2$.

It was shown in \cite[Theorem 6.1.1]{dt} that,
under certain restrictions on  $\psi$ and $w$,
$\omega_{\psi}^k(f,t)_{w,p}$ is equivalent to the following weighted $K$-functional $K_{k,\psi}(f,t^k)_{w,p}$:
\[
K_{k,\psi}(f,t^k)_{w,p}:=\inf_{g^{(k-1)}\in\AC_\loc}\left(
\|(f-g)w\|_{\L_p(D)} +t^k\|w\psi^kg^{(k)}\|_{\L_p(D)}\right).
\]

In particular, with  obvious modifications for $(-1,1)$ instead of $D=(0,1)$, $\psi := \varphi$ and $w := \varphi^r$, we have
\begin{eqnarray} \label{dtweighted}
&&\omega_{\varphi}^k(f,t)_{\varphi^r,p} \nonumber\\
&&\ =\sup_{0<h\leq
t}\norm{\varphi^r\Delta_{h\varphi}^kf}{\Lp[-1+t^*,1-t^*]}\\ \nonumber
&&\qquad +\sup_{0<h\leq
t^*}\norm{\varphi^r\overrightarrow{\Delta}_{h}^kf}{\Lp[-1,-1+At^*]} +
\sup_{0<h\leq t^*}\norm{\varphi^r
\overleftarrow{\Delta}_{h}^kf}{\Lp[1-At^*,1]},
\end{eqnarray}
where $t^* := 2k^2 t^2$ and $A$ is an absolute constant (for
example, $A=12$ as in \cite{dt}), and
note that
it is readily seen that the $K$-functional defined in Definition
\ref{k-func}, satisfies
\[
K_{k,r}^\varphi(f,t^k)_{p}=K_{k,\varphi}(f, t^k)_{\varphi^r
,p}=\inf_{g^{(k-1)}\in\AC_\loc} \left( \|(f-g)\varphi^r\|_p +
t^k\|\varphi^{k+r} g^{(k)}\|_p \right).
\]

 It follows from \cite[Theorem 6.1.1]{dt} that
 \be\label{equivalence}
M^{-1} \omega_{\varphi}^k(f,t)_{\varphi^r,p}
\leq K_{k,r}^\varphi(f,
t^k)_{p}\leq M\omega_{\varphi}^k(f,t)_{\varphi^r,p} \,, \ee for some
$M>1$ and   $0<t\leq t_0$.

A similar quantity to the following averaged modulus was considered
in \cite[(6.1.9)]{dt} (recall that $t^* := 2k^2 t^2$):
\ba\label{dtaveraged}
\omega_{\varphi}^{*k}(f,t)_{w,p} &=& \left(
\frac1t \int_0^t \int _{-1+t^*}^{1-t^*} |w(x) \Delta^k_{\tau
\varphi(x)}(f,x)|^p \, dx\, d\tau \right)^{1/p}\\ \nonumber && +
\left( \frac1{t^*} \int_0^{t^*} \int_{-1}^{-1+At^*}  |w(x)
\overrightarrow{\Delta}^k_{u}(f,x)|^p \, dx\, du \right)^{1/p}\\
\nonumber && + \left( \frac1{t^*} \int_0^{t^*} \int_{1-At^*}^1 |w(x)
\overleftarrow{\Delta}^k_{u}(f,x)|^p \, dx\, du \right)^{1/p},
\ea
where $1\leq p <\infty$.

Also, from the statement in \cite[p. 57]{dt}, we conclude that, for sufficiently small $t>0$,
\be \label{dtupper}
K_{k,r}^\varphi(f,t^k)_{p}\leq M_1 \omega_{\varphi}^{*k}(f,
t)_{\varphi^r,p}.
\ee

\sect{Proof of \thm{thm1.4}: the lower estimate}\label{lowerest}

We will apply \ineq{dtupper} (for $1\leq p <\infty$) and the second inequality in \ineq{equivalence} (for $p=\infty$) in order to complete the proof of the lower estimate in \thm{thm1.4}.

\begin{lemma} \label{auxlemmanew} Let   $k\in\N$, $r\in\N_0$, $1\leq p < \infty$ and $f\in\B_p^r$. Then
\[
\omega_{\varphi}^{*k}(f^{(r)},t)_{\varphi^r,p}\leq
c(k,r) \wkrav(f^{(r)},c(k)t)_p,\quad 0<t \leq c(k).
\]
\end{lemma}

\begin{proof}
We estimate each of the
three terms in the definition \ineq{dtaveraged} separately.

First, recall that $t^* := 2k^2 t^2$ and note that $[-1+t^*, 1-t^*]  \subset  \Dom_{2kt}$ and so using \prop{prop}{compar}
we have
\begin{eqnarray} \label{ineqaver}
\lefteqn{\frac1t\int_0^t\int_{-1+t^*}^{1-t^*}|\varphi^r(x)\Delta^k_{\tau\varphi(x)}
(f^{(r)},x)|^p\,dx\,d\tau}\\\nonumber&\leq &\frac{2^{rp}}{t}
\int_0^t\int_{-1+t^*}^{1-t^*}|\wt^r_{k\tau}(x)\Delta^k_{\tau
\varphi(x)}(f^{(r)},x)|^p \, dx\, d\tau \\ \nonumber&\leq&
\frac{2^{rp}}{t} \int_0^t \int _{\Dom_{2kt}} |\wt^r_{k\tau} (x)
\Delta^k_{\tau\varphi(x)}(f^{(r)},x)|^p\,dx\,d\tau\\\nonumber&\leq&
2^{rp}\wkrav(f^{(r)},t)_p^p\,.
\end{eqnarray}
We now estimate the second term (dealing with the function near $-1$), the third term being similar.

If $t$ is sufficiently small (for example, $t \leq (2k\sqrt{A+k/2})^{-1}$ will do),
then
\begin{eqnarray*}
\lefteqn{\frac1{t^*}\int_0^{t^*}\int_{-1}^{-1+At^*}|\varphi^r(x)\overrightarrow{\Delta}^k_{u}(f^{(r)},x)|^p\,dx\,du}\\
&=&\frac1{t^*}\int_0^{t^*}\int_{-1}^{-1+At^*}|\varphi^r(x)\Delta^k_{u}(f^{(r)},x+ku/2)|^p\,dx\,du\\
&\leq&\frac1{t^*}\int_0^{t^*}\int_{-1+ku/2}^{-1+(A+k/2)t^*}|\varphi^r(y-ku/2)\Delta^k_{u}(f^{(r)},y)|^p\,dy\,du\\
&\leq&\frac1{t^*}\int_{-1}^{-1+(A+k/2)t^*}\int_0^{2(y+1)/k}|\varphi^r(y-ku/2)\Delta^k_{u}(f^{(r)},y)|^p\,du\,dy\\
&=&\frac1{t^*}\int_{-1}^{-1+(A+k/2)t^*}\int_0^{2(y+1)/(k\varphi(y))}\varphi(y)|\varphi^r(y-kh\varphi(y)/2)
\Delta^k_{h\varphi(y)}(f^{(r)},y)|^p\,dh\,dy\\&\leq&c
\frac1{t^*}\int_{-1}^{-1+(A+k/2)t^*}
\int_0^{2(y+1)/(k\varphi(y))}\varphi(y)|\wt_{kh}^r(y)\Delta^k_{h\varphi(y)}(f^{(r)},y)|^p\,dh\,dy\\
&\leq&c\frac1{\sqrt{t^*}}\int_{-1}^{-1+(A+k/2)t^*}\int_0^{2(y+1)/(k\varphi(y))}|\wt_{kh}^r(y)
\Delta^k_{h\varphi(y)}(f^{(r)},y)|^p\,dh\,dy\\&\leq&c
\frac1{\sqrt{t^*}}\int_0^{c\sqrt{t^*}}\int_{\Dom_{kh}\cap[-1,-1+(A+k/2)t^*]}
|\wt_{kh}^r(y) \Delta^k_{h\varphi(y)}(f^{(r)},y)|^p \,dy\,dh\\
&\leq&c\wkrav(f^{(r)},c(k)t)_p^p,
\end{eqnarray*}
where, for the third inequality,   
we used the fact that $\varphi (y-kh\varphi(y)/2) \leq \sqrt{2} \wt_{kh} (y)$ if $0\leq h \leq 2(y+1)/(k\varphi(y))$ and $y\leq -1/2$.
\end{proof}

\begin{lemma} \label{auxlemmainfty}
Let   $k\in\N$, $r\in\N_0$  and $f\in\B_\infty^r$. Then
\[ \omega_{\varphi}^{k}(f^{(r)},t)_{\varphi^r,\infty} \leq c(k,r) \wkr(f^{(r)},
c(k)t)_\infty,\quad 0< t \leq c(k). \]
\end{lemma}

\begin{proof}
The proof is very similar to that of \lem{auxlemmanew}.
First,   recalling that $t^* := 2k^2 t^2$, noting that $[-1+t^*, 1-t^*]  \subset  \Dom_{2kt}$ and   using \prop{prop}{compar}
we have
\[
  \sup_{0<h\leq t} \norm{\varphi^r(\cdot) \Delta_{h\varphi(\cdot)}^k (f^{(r)}, \cdot)}{\L_\infty[-1+t^*,1-t^*]} \leq 2^{r }\wkr(f^{(r)},t)_\infty .
\]

If $0<h\le t^*$, then
\begin{eqnarray*}
\lefteqn{ \left\| \varphi^r(\cdot)  \overrightarrow{\Delta}_{h}^k (f^{(r)}, \cdot) \right\|_{\L_\infty [-1,-1+At^*]}  }\\
   & =&
\sup_{x\in[-1,-1+At^*]} \left|\varphi^r(x) \Delta^k_h(f^{(r)},x+kh/2) \right| \\
&\le& \sup_{y\in[-1+kh/2,-1+(A+k/2) t^*]} \left| \varphi^r(y-kh/2)\Delta^k_h(f^{(r)},y)\right| .
 \end{eqnarray*}

Hence,
if $t$ is sufficiently small ($t \leq (2k\sqrt{A+k/2})^{-1}$ will do),
then
\begin{eqnarray*}
\lefteqn{ \sup_{0<h\leq t^*} \left\| \varphi^r(x)  \overrightarrow{\Delta}_{h}^k (f^{(r)}, x) \right\|_{\L_\infty [-1,-1+At^*]} }\\
& \leq & \sup_{y\in[-1,-1+(A+k/2) t^*]} \sup_{0<h\leq 2(y+1)/k} \left| \varphi^r(y-kh/2)\Delta^k_h(f^{(r)},y)\right| \\
& =  & \sup_{y\in[-1,-1+(A+k/2) t^*]} \sup_{0<h\leq 2(y+1)/(k\varphi(y)}  \left| \varphi^r(y-kh\varphi(y)/2)\Delta^k_{h\varphi(y)} (f^{(r)},y)  \right|  \\
&\le &
2^{r/2} \sup_{y\in[-1,-1+(A+k/2) t^*]} \sup_{0<h\leq 2(y+1)/(k\varphi(y)}  \left| \wt_{kh}^r(y)\Delta^k_{h\varphi(y)} (f^{(r)},y)  \right|  \\
&\le &
2^{r/2} \sup_{0<h\leq c(k)t } \sup_{y\in\Dom_{kh} \cap [-1,-1+(A+k/2) t^*]}   \left| \wt_{kh}^r(y)\Delta^k_{h\varphi(y)} (f^{(r)},y)  \right|  \\
& \leq & 2^{r/2} \wkr(f^{(r)},c(k) t)_\infty .
\end{eqnarray*}

 The estimate of $ \sup_{0<h\leq t^*}\norm{\varphi^r(\cdot) \overleftarrow{\Delta}_{h}^k (f^{(r)}, \cdot)}{\L_\infty[1-At^*,1]}$ is similar.
\end{proof}

We are now ready to complete the proof of the lower estimate in \thm{thm1.4}.
First, estimates \ineq{equivalence} and \ineq{dtupper} together with Lemmas~\ref{auxlemmanew} and \ref{auxlemmainfty} imply that, for $f\in \B^r_p$, $1\leq p \leq \infty$,
\be\label{inequ}
K_{k,r}^\varphi(f^{(r)},t^k)_{p} \leq c\wkrav(f^{(r)}, c_1  t)_p , \quad 0<t \leq c_2  ,
\ee
where $c_1=c_1(k)$ and $c_2=c_2(k)$ are some positive constants which we now consider fixed.

Now, suppose that $0<t\leq 2/k$ and let $\mu := \max\{1, c_1, 2/(kc_2) \}$. Then, since $t/\mu \leq c_2$, taking into
account \ineq{ineqavermon}, we have
 \begin{eqnarray*}
K_{k,r}^\varphi(f^{(r)},t^k)_{p} & \leq & \mu^k K_{k,r}^\varphi(f^{(r)},(t/\mu)^k)_{p} \leq c  \wkrav(f^{(r)}, c_1  t/\mu)_p \\
& \leq & c (\mu/c_1)^{1/p} \wkrav(f^{(r)}, t)_p ,
\end{eqnarray*}
which completes the proof of the lower estimate in \thm{thm1.4}.

\sect{Hierarchy between  moduli}

The following theorem  illustrates the hierarchy between the
moduli of smoothness.
\begin{theorem}\label{hierarchy}
If $f\in\B^{r+1}_p$, $1\le p \le \infty$, $r\in\N_0$ and $k\ge2$, then
\[
\omega_{k,r}^\varphi(f^{(r)},t)_p \leq
ct\w_{k-1,r+1}^\varphi(f^{(r+1)},t)_p .
\]
\end{theorem}
\begin{proof}
By virtue of \cite[(6.2.9)]{dt}, we have
\[
 \omega_\varphi^k(f^{(r)},t)_{\varphi^r,p}\le
c\int_0^t(\Omega^k_\varphi(f^{(r)},\tau)_{\varphi^r,p}\,/\tau)\,d\tau,
\]
where $\Omega^k_\varphi$ was defined in \cite[(8.1.2)]{dt} as follows
\[  
\Omega^k_\varphi(f,t)_{\varphi^r,p}:=\sup_{0<h\leq t}
\norm{\varphi^r \Delta_{h\varphi}^kf}{\Lp[-1+2k^2h^2 , -1+2k^2h^2]}.
\]

Also, by
\cite[(6.3.2)]{dt}, we obtain
\[
\Omega^k_\varphi(f^{(r)},t)_{\varphi^r,p}\le
ct\Omega^{k-1}_\varphi(f^{(r+1)},t)_{\varphi^{r+1},p}.
\]
Hence,
\begin{align*}
 \omega_\varphi^k(f^{(r)},t)_{\varphi^r,p}
&\le
c\int_0^t\Omega^{k-1}_\varphi(f^{(r+1)},\tau)_{\varphi^{r+1},p}\,d\tau\\
&\le
ct\Omega^{k-1}_\varphi(f^{(r+1)},t)_{\varphi^{r+1},p}\le
ct \omega^{k-1}_\varphi
(f^{(r+1)},t)_{\varphi^{r+1},p},
\end{align*}
where for the second inequality we used the monotonicity of
$\Omega^{k-1}_\varphi(f^{(r+1)},t)_{\varphi^{r+1},p}$, and for the
third we applied \cite[(6.2.9)]{dt}.

In view of the equivalence between
our and weighted DT moduli
(see Remark~\ref{remequivdt}),
our proof is
complete.
\end{proof}

We also have the other usual hierarchy which follows by \cite[Theorem 6.1.4]{dt}.

\begin{theorem}\label{hierarchy2}
If $f\in\B^r_p$, $1\le p\le\infty$, $r\in\N_0$ and $k\ge2$, then
\[
\omega_{k,r}^\varphi(f^{(r)},t)_p \leq
c \w_{k-1,r}^\varphi(f^{(r)},t)_p .
\]
\end{theorem}

\sect{Polynomial approximation: direct results\label{sec5}}

This section is devoted to the approximation of
functions $f\in L_p[-1,1]$, $1\le p\le\infty$, by polynomials of degree $<n$. Let
$\mathcal P_n$ be the set of polynomials of degree $<n$ and denote
by
\[
E_n(f)_p=\inf_{p_n\in\mathcal P_n}\|f-p_n\|_p,
\]
the degree of approximation of $f\in L_p[-1,1]$ by elements of
$\mathcal P_n$.

An immediate application of Theorem \ref{hierarchy}, together with
\cite[Theorem 7.2.1]{dt}, is the following.

\begin{theorem}\label{estima} If $f\in\B^r_p$, $1\le p\le\infty$, then
\be \label{estimate}
E_n(f)_p\le\frac c{n^r}\wkr(f^{(r)},1/n)_p,\quad n\ge k+r.
\ee
\end{theorem}
\begin{proof}
It follows from \cite[Theorem 7.2.1]{dt} that
\[
E_n(f)_p\le c\omega_{k+r}^\varphi(f,1/n)_p,\quad n\ge k+r.
\]
Since $f\in\B^r_p$, we apply Theorem \ref{hierarchy} $r$
times and \eqref{estimate} follows.
\end{proof}

An immediate consequence of \thm{estima} is the following direct estimate.

\begin{corollary} \label{corestim}
If $f\in\B^r_p$, $r\in\N_0$, $1\le p\le\infty$, and if for some $k\in\N$,
and $\alpha>r$, $\wkr(f^{(r)},t)_p=O(t^{\alpha-r})$, then
\be\label{estim}
E_n(f)_p\le cn^{-\alpha},\quad n\ge k+r.
\ee
\end{corollary}

It is interesting to compare \ineq{estimate} with estimates of how well $P_n^{(r)}$ approximates $f^{(r)}$.
Our result here is the following.

\begin{theorem} If $f\in\B^r_p$, $r\in\N$, $1\le p\le\infty$, and $P_n$ denotes the
polynomial of best approximation of $f$ in $L_p[-1,1]$, of degree $<n$. If $\int_0^1\bigl(\wkr(f^{(r)},\tau)_p/\tau\bigr)d\tau<\infty$, for some $k\in\N$, then
\[
\|(f^{(r)}-P_n^{(r)})\varphi^r\|_p\le c\int_0^{1/n}\bigl(\wkr(f^{(r)},\tau)_p/\tau\bigr)d\tau.
\]
\end{theorem}
\begin{proof}
Using Potapov's estimate (see \eg \cite[(7.2.7)]{dt})
\be \label{potap}
\|\varphi^\nu P_n^{(\nu)} \|_p \leq c(p,\nu)   n^\nu \|P_n\|_p ,
\ee
we have
\begin{eqnarray*}
\|(f^{(r)}-P_n^{(r)})\varphi^r\|_p&\le&
\sum_{j=1}^\infty\|(P_{2^jn}^{(r)}-P_{2^{j-1}n}^{(r)})\varphi^r\|_p\\
&\le& c \sum_{j=1}^\infty2^{jr}n^r\|P_{2^jn}-P_{2^{j-1}n}\|_p\\
&\le& c\sum_{j=1}^\infty2^{jr}n^r\left(\|P_{2^jn}-f\|_p+
\|f-P_{2^{j-1}n}\|_p\right)\\
&\le&c  \sum_{j=1}^\infty\wkr(f^{(r)},1/(2^{j}n))_p\\
&\le & c \sum_{j=1}^\infty
\int_{1/(2^{j}n)}^{1/(2^{j-1}n)} \bigl(\wkr(f^{(r)},\tau)_p/\tau\bigr)d\tau \\
&\le&c  \int_0^{1/n}
\bigl(\wkr(f^{(r)},\tau)_p/\tau\bigr)d\tau,
\end{eqnarray*}
where for the first inequality we used the fact that $\|(f^{(r)}-P_{2^jn}^{(r)})\varphi^r\|_p\to0$ (see details in the proof of \thm{Theorem 3.199} below),
and for the fourth inequality we used \thm{estima} and \cor{maincorol}.
\end{proof}

\sect{Polynomial approximation:  inverse theorems\label{sec8}}

Denote by $\Phi$  the set of nondecreasing functions
$\phi:[0,1]\rightarrow[0,\infty)$, satisfying $\phi(0+)=0$.
Recalling that $E_n(f)_p$ is the degree of approximation of $f$ by polynomials of degree $<n$, we have the following inverse theorem.

\begin{theorem} \label{Theorem 3.199}
Given $1\leq p \leq \infty$, $k\in\N$, $r\in\N_0$, $N\in\N$, and $\phi\in\Phi$ such that
$$
 \int_0^1 \frac{r \phi (u)}{u^{r +1}}du<+\infty
$$
(\ie  if $r=0$, this condition is not needed).
If
$$
E_n(f)_p\le\phi\left(\frac1n\right),\quad \mbox{for all} \quad n\ge N,
$$
then one of the representatives of $f$ has a locally absolutely continuous derivative $f^{(r-1)}$,
$f\in  \B^r_p$, and
\[
\omega_{k,r}^\varphi( f^{(r)},t)_p\le
c\int_0^t\frac{r\phi(u)}{u^{r+1}}du+ct^{k}\int_t^1\frac{\phi(u)}{u^{k+r+1}}du
+c(N)t^{k}E_{k+r}(f)_p,\quad t\in\left[0,1/2\right].
\]
If, in addition, $N\le k+r$, then
\[
\omega_{k,r}^\varphi(f^{(r)},t)_p\le
c \int_0^t\frac{r\phi(u)}{u^{r+1}}du+ct^k\int_t^1\frac{\phi(u)}{u^{k+r+1}}du,\quad t\in\left[0,1/2\right].
\]
\end{theorem}

\begin{remark}
For $p=\infty$, this theorem was proved in \cite{kls3}.
\end{remark}

\begin{remark}
A  classical restatement of \thm{Theorem 3.199} is \thm{primetheorem} below. We prefer the current (integral) version since we find it rather more convenient to use.
\end{remark}

\begin{proof}
Since the theorem was proved in \cite{kls3} for $p=\infty$, we may assume that $1\leq p <\infty$.

 We first give the
proof for the case $r\ge 1$. Without any loss of generality assume
that $N\ge k+r$. Set $m_j:=N2^j$ and $\phi_j:=\phi(m_j^{-1})$. We
represent $f$ as the telescopic series (converging to $f$ in  $\Lp$)
\be\label{tele}
f = P_{k+r}+(P_{N}-P_{k+r})+\sum_{j=0}^\infty(P_{m_{j+1}}-P_{m_{j}})
=:P_{k+r}+Q+ \sum_{j=0}^\infty Q_j, \ee
where $P_n\in\Pn$ are
the polynomials of best approximation of $f$, that is
$\|f-P_n\|_p=E_n(f)_p$. Hence, the polynomials $Q_j$ are of degree
$<m_{j+1}$ and satisfy $\|Q_j\|_p\le \phi_{j+1}+\phi_j\le 2\phi_j$ .

As in the proof of \lem{less},
for each  $0\leq i\leq k$,   put
$y_i(x):=x+(i-k/2)h\varphi(x)$ and recall that   $\wt_\delta(x)\leq\varphi(y)$ for all $x\in\Dom_\delta$,
$y\in \left[x-\delta\varphi(x)/2,x+\delta\varphi(x)/2\right]$ and
$0<\delta \leq 2$.

Then, using \ineq{potap}, we have
\begin{eqnarray*}
 \|\wt_{kh}^r \Delta_{h\varphi}^k(Q_j^{(r)}, \cdot) \|_p
& \le &
\norm{\sum_{i=0}^k {\binom ki}
\varphi^r(y_i)\left|Q_j^{(r)}(y_i)\right|}{\Lp(\Dom_{kh})}
 \le  2^{k+1/p}\norm{\varphi^r Q_j^{(r)} }{p}\\
 & \le& c m_{j+1}^r \|Q_j\|_p \leq c m_{j}^r \phi_j .
\end{eqnarray*}

Therefore, if we denote $J:=\min\{j:1/{m_j}\le t\}$, then we have
\begin{eqnarray}\label{app1}
 \wkr\left( \sum_{j=J+1}^\infty Q_j^{(r)}, t\right)_p
 & \le & c
\sum_{j=J+1}^\infty
m_j^{r }\phi_j
 \le c \sum_{j=J+1}^\infty
\int_{m_j^{-1}}^{m_{j-1}^{-1}}\frac{\phi_j }{u^{r +1}}du \\ \nonumber
& \le &
c \sum_{j=J+1}^\infty\int_{m_j^{-1}}^{m_{j-1}^{-1}}
\frac{\phi (u)}{u^{r +1}}du
 \le c \int_{0}^{m_{J}^{-1}}\frac{\phi (u)}{u^{r +1}}du \\ \nonumber
&\le&
c \int_{0}^{t}\frac{\phi (u)}{u^{r +1}}du\,.
\end{eqnarray}

 We also note that, in a similar fashion, \ineq{potap} implies that
 \begin{eqnarray*}
 \sum_{j=0}^\infty\|\varphi^{\nu}Q_j^{(\nu)}\|_p&\le&c\sum_{j=0}^\infty m_j^{\nu}\phi_j\le c\int_0^1\frac{\phi(u)}{u^{\nu +1}}du < \infty ,
 \end{eqnarray*}
 for all $1\leq \nu\leq r$. This implies that $\{ \varphi^\nu P_{m_k}^{(\nu)} \}$ is a Cauchy sequence in $\Lp$ converging to some function $\varphi^\nu f_\nu \in\Lp$, and there is a subsequence of this sequence which converges pointwise almost everywhere to $\varphi^\nu f_\nu$. Moreover, we conclude that there exists a sequence $\{ n_l\} \subset \N$ such that, for each $1\leq \nu \leq r$, $\{  P_{n_l}^{(\nu)} \}$ converges pointwise almost everywhere to $ f_\nu$ and
 \be \label{timan}
 \lim_{l \to \infty} \|\varphi^\nu (f_\nu - P_{n_l}^{(\nu)})\|_p = 0 .
  \ee
  Now, considering $S := [-1+\e, 1-\e]$, $\e>0$, denoting $f_0 := f$ and using the argument \cite[section 6.1.3]{ti-book}, we write, for $x_0$ which is one of the points of convergence for all $1\leq \nu \leq r$,
 \begin{eqnarray*}
 \lefteqn{ f_{\nu-1}(x) - f_{\nu-1}(x_0) - \int_{x_0}^x f_\nu(t) dt}\\
 & =&
 f_{\nu-1}(x) - P_{n_l}^{(\nu-1)}(x)   - \left( f_{\nu-1}(x_0)-P_{n_l}^{(\nu-1)}(x_0)\right)  - \int_{x_0}^x \left( f_\nu(t)-P_{n_l}^{(\nu)}(t)\right)  dt ,
 \end{eqnarray*}
 and conclude that
 \[
 f_{\nu-1}(x) - f_{\nu-1}(x_0) = \int_{x_0}^x f_\nu(t) dt ,
 \]
 for almost all $x\in S$ and all $1\leq \nu \leq r$. 
 Hence, recalling that $f_0 = f$, it follows that almost everywhere $f(x)$ is identical with a function possessing an absolutely continuous derivative of order $(r-1)$ and $f^{(r)} \in \Lp(S)$. Hence, differentiation of \ineq{tele} is justified. Also, \ineq{timan} implies that $f\in \B_p^r$.

We now continue with our estimates, and using \ineq{potap}   with  $\nu=r+k$ we have
\[
\| \varphi^{r+k} Q_j^{(r+k)} \|_p \leq c m_{j+1}^{r+k} \| Q_j\|_p \leq c m_j^{r+k} \phi_j .
\]
Hence, for  $0\leq j\leq J$, taking into account that $1/{m_j}>t/2$ and denoting $m_{-1} := N/2$, we have
 \begin{eqnarray} \label{app2}
 \lefteqn{ \wkr\left( \sum_{j=0}^{ J}  Q_j^{(r)}, t\right)_p }\\ \nonumber
 & \le & c t^{k } \sum_{j=0}^{J} \| \varphi^{r+k} Q_j^{(r+k)} \|_p
   \le   c t^{k} \sum_{j=0}^{J} m_j^{r+k} \phi_j
   \le  c t^{k} \sum_{j=0}^{J} \int_{m_j^{-1}}^{m_{j-1}^{-1}}\frac{\phi_j}{u^{k+r+1}}du \\ \nonumber
 &  \le &  c t^{k} \sum_{j=0}^{J} \int_{m_j^{-1}}^{m_{j-1}^{-1}}\frac{\phi(u)}{u^{k+r+1}}du
   \le   c t^{k}  \int_{m_J^{-1}}^{2/N}\frac{\phi(u)}{u^{k+r+1}}du
  \le   c t^{k}  \int_{t/2}^{1}\frac{\phi(u)}{u^{k+r+1}}du\\ \nonumber
   & \le & c t^{k}  \int_{t}^{1}\frac{\phi(u)}{u^{k+r+1}}du .
\end{eqnarray}

Finally, we have the estimate
 \be\label{app3}
\wkr(Q^{(r)}, t)_p \le c t^{k} \|\varphi^{k+r} Q^{(k+r)}\|_p \le c t^{k} N^{r+k} \|Q\|_p \le c t^{k} N^{r+k} E_{k+r}(f)_p .
\ee
 Note that if $N=k+r$, then $Q\equiv0$, so that the
left hand side of \ineq{app3} vanishes and no estimate is needed.

Now, the observation that
$\Delta^k_{h\varphi(x)}(P_{k+r}^{(r)},x)=0$, combined with
\ineq{app1}, \ineq{app2}, and \ineq{app3}, completes the proof of
the theorem for $r\ge1$.

For $r=0$, we write
$$
f=P_k+Q+\sum_{j=0}^JQ_j+(f-P_{m_{J+1}}),
$$
where $Q:=P_N-P_k$ and $Q_j:=P_{m_{j+1}}-P_{m_{j}}$ (see
\ineq{tele}), and complete the proof as above, just applying
\ineq{app2}, \ineq{app3}, and
\[
\|f-P_{m_{J+1}}\|_p = E_{m_{J+1}}(f)_p \leq \phi_{m_{J+1}}
\]
(\ie the same type of estimate as for $\|Q_{m_{J+1}}\|_p$).
\end{proof}

Choosing
\[
\phi(u) :=
\begin{cases}
E_n(f)_p , & \mbox{\rm if } 1/n \leq u < 1/(n+1) , \; n\geq N-1, \\
E_N(f)_p ,  & \mbox{\rm if } 1/N \leq u \leq 1 ,
\end{cases}
\]
in \thm{Theorem 3.199} we immediately get the following result which, in fact, is equivalent to \thm{Theorem 3.199}.

\renewcommand{\thesametheorem}{\ref{Theorem 3.199}${}^\prime$}

\begin{sametheorem} \label{primetheorem}
Given $1\leq p \leq \infty$, $k\in\N$, $r\in\N_0$, $N\in\N$.
If
$$
\sum_{n=1}^\infty  r n^{r-1} E_n(f)_p <+\infty
$$
(\ie  if $r=0$, this condition is not needed),
then one of the representatives of $f$ has a locally absolutely continuous derivative $f^{(r-1)}$,
$f\in  \B^r_p$, and
\begin{eqnarray*}
\omega_{k,r}^\varphi( f^{(r)},t)_p &\le&
c \sum_{n>\max\{N,1/t\}}rn^{r-1}E_n(f)_p\\&&\quad+ct^{k}\sum_{N\leq n\leq\max\{1/t, N\}}n^{k+r-1}E_n(f)_p
\\&&\quad+c(N)t^{k}E_{k+r}(f)_p,\quad t\in\left[0,1/2\right].
\end{eqnarray*}
If, in addition, $N\le k+r$, then
\begin{eqnarray*}
\omega_{k,r}^\varphi(f^{(r)},t)_p&\le&
c\sum_{n >\max\{N,1/t\}}r n^{r-1}E_n(f)_p\\&&\quad+ct^{k}\sum_{N\leq n\leq\max\{1/t,N\}}n^{k+r-1}E_n(f)_p,\quad t\in\left[0,1/2\right].
\end{eqnarray*}
\end{sametheorem}

Another immediate corollary of \thm{Theorem 3.199} with $\phi(t) := t^\alpha$ and $N=k+r$ is the following result which is an inverse to \ineq{estim} .

\begin{corollary}\label{r+k} Let $r\in\N_0$, $k\in\N$ and
  $r<\alpha<r+k$, and let $f\in L_p[-1,1]$, $1\leq p \leq \infty$. If
\be\label{aleph}
 E_n(f)_p\le n^{-\alpha}, \quad n\ge N, \ee for some
$N\ge k+r$, then $f\in\B^r_p$ and
\[
\wkr(f^{(r)},t)_p\le
c(\alpha,k,r)t^{\alpha-r}+c(N,k,r)t^kE_{k+r}(f)_p,\quad t>0.
\]
In particular, if  $N = k+r$, then $\ineq{aleph}$ implies that
 $f\in\B^r_p$ and
\[
\wkr(f^{(r)},t)_p\le c(\alpha,k,r)t^{\alpha-r},\quad t>0.
\]
\end{corollary}

Corollaries~\ref{corestim} and \ref{r+k} imply the following constructive characterization result.

\begin{corollary}
Let $r\in\N_0$, $k\in\N$, $r<\alpha<r+k$, and let $f\in L_p[-1,1]$, $1\leq p\leq \infty$.
Then
$E_n(f)_p \leq c n^{-\alpha}$, for all $ n\ge k+r$,
if and only if $f\in\B^r_p$ and
$\wkr(f^{(r)},t)_p\le c t^{\alpha-r}$, $t>0$.
\end{corollary}

\end{document}